\documentclass[12pt]{amsart}
\usepackage{amssymb}
\usepackage{eucal}
\usepackage[all]{xy}
\usepackage{shadow}
\usepackage{times}
\usepackage{enumerate}
\numberwithin{equation}{section}

\newtheorem{theorem}{Theorem}[section]
\newtheorem{proposition}[theorem]{Proposition}
\newtheorem{lemma}[theorem]{Lemma}

\theoremstyle{definition}
\newtheorem{definition}[theorem]{Definition}

\theoremstyle{remark}
\newtheorem{remark}[theorem]{Remark}

\newcommand{\ShortExactSeq}[3]%
{0\to{#1}\to{#2}\to{#3}\to 0}
\newcommand{\DShortExactSeq}[3]%
{\xymatrix@1{0\ar[r]&{{#1}}\ar[r]&{{#2}}\ar[r]&{{#3}}\ar[r]&0}}

\newcommand{\Ahat}{\widehat{\mathcal A}}

\newcommand{\Bounded}{{\frak B}}

\newcommand{\Index}{\operatorname{\rm Index}}

\newcommand{\R}{\mbox{$\Bbb R$}}

\newcommand{\Supp}{\operatorname{\rm Supp}}

\newcommand{\Z}{\mbox{$\Bbb Z$}}

\renewcommand{\ge}{\geqslant}
\renewcommand{\le}{\leqslant}
 \renewcommand{\phi}{\varphi}
\renewcommand{\epsilon}{\varepsilon}

\newcounter{ritmctr}
{\end{itemize}}
\newcounter{aitmctr}
{\end{itemize}}

\begin{document}
\bibliographystyle{plain}

\title{Positive curvature, partial vanishing theorems and coarse indices}

\author{John Roe}
\address{Department of Mathematics, Penn State University, University
Park PA 16802}
\date{\today}
\maketitle

\section{Introduction}
Let $M$ be a complete Riemannian manifold and let $D$ be a generalized Dirac operator acting on sections of a Clifford bundle $S$ over $M$.  It is well-known (see for example~\cite{GrLa}) that there is a Weitzenbock formula
\[ D^2 = \nabla^*\nabla + R , \]
where $R$ is a certain self-adjoint endomorphism of $S$ constructed out of the curvature.  (For example, in the classical case of the Dirac operator associated to a spin-structure, $R$ is pointwise multiplication by $\frac14$ times the scalar curvature~\cite{Li}).

The author's coarse index theory associates to $D$ an index that lies in the $K$-theory of the ``translation $C^*$-algebra'' $C^*(M)$.  As in the classical case, the index vanishes if the curvature operator is uniformly bounded below by a positive constant.  In \cite[Proposition 3.11]{JR18} this statement is generalized as follows.  Suppose that there is a subset $Z\subseteq M$, such that for some constant $a>0$ one has $R_x \ge a^2I$ (as self-adjoint endomorphisms of $S_x$) for all $x\notin Z$ --- we will then say that the operator $R$ is \emph{uniformly positive outside $Z$}.   Then the index of $D$ lies in the image of the map
\[ K_*(C^*(Z)) \to K_*(C^*(X)), \]
where $Z$ is considered as a metric subspace of $X$.  In particular, if the curvature is uniformly positive outside a compact set $Z$ (so that $C^*(Z)$ is the compact operators), one recovers the result of Gromov and Lawson~\cite[Chapter 3]{GrLa} that $D$ has an index in the ordinary Fredholm sense.

I included only the briefest sketch of a proof of this proposition in~\cite{JR18}.  This note is a response to several requests for more detail, and also mentions a couple of applications of the idea.

\section{The main result}
Let $Z$ be a subset of a proper metric space $X$ and let $H$ be an ample $X$-module (i.e. a Hilbert space which is a ``sufficiently large'' module over $C_0(X)$, assumed fixed --- I refer to~\cite{HR3} for terminology.  The module action is denoted by $\rho\colon C_0(X)\to\Bounded(H)$.).  The $C^*$-algebra $C^*(X)$ (or $C^*(X;H)$ if it is important to keep track of the particular Hilbert space) is then defined to be the norm closure of the controlled, locally compact operators on $H$, where we recall that a \emph{controlled} (a.k.a. \emph{finite propagation}) operator $T$ has the property that there is a constant $r$ for which
\[ d(\Supp \phi,\Supp \psi)>r \Longrightarrow \rho(\phi) T \rho(\psi) = 0 \]
for all $\phi,\psi \in C_c(X)$.  A controlled operator $T$ is \emph{supported near $Z$} if there is another constant $r'$ for which
\[ d(\Supp\phi, Z) > r' \Longrightarrow \rho(\phi) T = 0 = T \rho(\phi). \]
The norm closure of the set of controlled, locally compact operators supported  near $Z$ is an ideal in $C^*(X)$, which we denote\footnote{It is denoted $C^*_X(Z)$ in~\cite{JR18}, but the other notation now seems better to me.} by  $C^*(Z\subseteq X)$. It is easy to see~\cite{HRY} that the $K$-theory of $C^*(Z\subseteq X)$ is the same as that of $C^*(Z)$, if $Z$ is considered as a metric space in its own right.

Now we recall the relation of these concepts to index theory.  Suppose that $X$ is actually a complete Riemannian manifold and that $S$ is a Clifford bundle, and let $H=L^2(X;S)$ in forming the algebras above.  The algebra $D^*(X)$ is defined to be the norm closure of the controlled, \emph{pseudolocal} operators on $H$: it is a unital $C^*$-algebra, and $C^*(X)$ is an ideal in it.  The following key analytic lemma~\cite[Chapter 10]{HR3} can be proved by the finite propagation speed method:

\begin{lemma}\label{L1} Let $X$ be a complete Riemannian manifold, as above, and let $S$ be a Clifford bundle over it.  Let $D$ denote the Dirac operator of $S$, considered as an unbounded, self-adjoint operator on $H=L^2(X;S)$. If $f$ is a bounded continuous function on $\R$ that has finite limits at $\pm\infty$, then $f(D)\in D^*(X;H)$.  If $f$ tends to zero at $\pm\infty$, then $f(D)\in C^*(X;H)$.  \quad\qedsymbol \end{lemma}

A \emph{normalizing function} $\chi\colon \R \to [-1,1]$ is, by definition, a continuous, odd function that tends to $\pm 1$ at $\pm\infty$.  Given such a function $\chi$, it follows from the preceding lemma that $\chi(D)\in D^*(X)$ and $\chi(D)^2-1\in C^*(X)$.  Moreover, if $\chi_1$ and $\chi_2$ are two normalizing functions, then it similarly follows that $\chi_1(D)-\chi_2(D)\in C^*(X)$.   Thus the equivalence class of $\chi(D)$ gives a well-defined  self-adjoint involution in $D^*(X)/C^*(X)$, defining an element $[\chi(D)]\in K_{j+1}(D^*(X)/C^*(X))$ ($j$ is determined by the grading of the operator --- it is equal to the parity of $\dim X$).  Now we have

\begin{definition} With the notation of Lemma~\ref{L1}, the \emph{coarse index} of $D$ is
\[ \Index(D) = \partial[\chi(D)] \in K_j(C^*(X)), \]
where $\partial\colon K_{j+1}(D^*(X)/C^*(X)) \to K_j(C^*(X))$ is the boundary map in the long exact sequence of $C^*$-algebra $K$-theory. \end{definition}

Now let $Z\subseteq X$ as above.  The algebra $C^*(Z\subseteq X)$ is an ideal in $D^*(X)$ (not just in $C^*(X)$).   To prove our result we will need to sharpen Lemma~\ref{L1} as follows:

\begin{lemma}\label{L2} Let notation be as in Lemma~\ref{L1}.  Suppose that the curvature operator $R=R_D$ that appears in the Weitzenbock formula for $D$,
 \[ D^2 = \nabla^*\nabla + R_D, \]
 is uniformly positive outside $Z$, say $R_x\ge a^2I $ for $x\notin Z$.  Then for any $f\in C_c(-a,a)$ we have $f(D) \in C^*(Z\subseteq X)$. \end{lemma}

Suppose that this lemma has been proved.  Then choose a normalizing function $\chi$ such that $\chi^2-1$ is supported in $(-a,a)$.  According to Lemma~\ref{L2}, the equivalence class of $\chi(D)$ is a (well-defined) self-adjoint involution in $D^*(X)/C^*(Z\subseteq X)$.  Following the construction above, we obtain a \emph{localized index}
\[ \Index_Z(D) \in K_j(C^*(Z\subseteq X)) \]
which maps to the previously defined $\Index(D)$ under the K-theory map induced by the inclusion $C^*(Z\subseteq X)\to C^*(X)$.  The existence of this localized index is the precise content of \cite[Proposition 3.11]{JR18}; it implies the version of the result stated in the introduction.  To state it precisely:

\begin{theorem}\label{LCA} Let $M$ be a complete Riemannian   and let $D$ be a Dirac-type operator whose associated curvature endomorphism $R_D$ is uniformly positive outside a subset $Z$ of $M$.  Then the construction above defines a \emph{localized coarse index}
\[ \Index_Z(D) \in K_j(C^*(Z\subseteq X)) \]
which maps to the coarse index $\Index(D)\in K_j(C^*(M))$ under the K-theory map induced by the inclusion $C^*(Z\subseteq X)\to C^*(X)$.  (Here $j$ is the parity of $\dim M$.) \end{theorem}

The rest of this section will give the proof of Lemma~\ref{L2}.  In order to use the finite propagation speed method, we consider first the properties of functions $f$ that have compactly supported Fourier transforms.

\begin{lemma}\label{L3a} With notation as in Lemma~\ref{L2}, suppose that $f\in {\mathcal S}(\R)$ is an even function and has Fourier transform $\hat{f}$ supported in $(-r,r)$.   Let $\phi\in C_0(X)$ have support disjoint from a $2r$-neighborhood of $Z$.  Then
\[ \|f(D) \rho(\phi)\| \le \|\phi\|\sup\{ |f(\lambda)| : |\lambda|\ge a \} \]
and the same estimate applies to $\rho(\phi)f(D) $. \end{lemma}

\begin{proof}  We   use the Fourier cosine formula
\[ f(D) = \frac{1}{\pi}\int_0^r \hat{f}(t) \cos(tD)\, dt, \]
remembering that $\hat{f}(t)$ vanishes for $t>r$.  Now let $U_n= \{x\in X: d(x,Z)>nr\}$, for $n=1,2$, and consider the unbounded, symmetric operator which is equal to $D^2$ with domain $C^\infty_c(U_1)$.  This operator is bounded below by $a^2I$ and therefore it has a Friedrichs extension on the Hilbert space $L^2(U_1;S)$ which is also bounded below (with the same bound) and which we shall denote by $E$.

A standard finite propagation speed argument shows that if $s$ is smooth and compactly supported in $U_2$ then
\[ \cos(tD)s = \cos(t\sqrt{E})s,\quad \mbox{for $0\le t\le r$}.\]
In particular, $\cos(tD)\rho(\phi) = \cos(t\sqrt{E})\rho(\phi)$ for these values of $t$.  Via the Fourier integral above, this implies that $f(D)M_\phi = f(\sqrt{E})M_\phi$.
But since the spectrum of $\sqrt{E}$ is bounded below by $a$,
\[ |f(\sqrt{E})| \le \sup\{ |f(\lambda)| : |\lambda|\ge a \}, \]
and this gives the desired estimate.
\end{proof}

There is a version of Lemma~\ref{L3a} without the evenness hypothesis.

\begin{lemma}\label{L3} With notation as above, suppose that $f\in {\mathcal S}(\R)$   has Fourier transform $\hat{f}$ supported in $(-r,r)$.   Let $\phi\in C_0(X)$ have support disjoint from a $4r$-neighborhood of $Z$.  Then
\[ \|f(D) \rho(\phi)\| \le 2\|\phi\|\sup\{ |f(\lambda)| : |\lambda|\ge a \} \]
and the same estimate applies to $\rho(\phi)f(D) $. \end{lemma}

\begin{proof} If $f$ is even, this is a consequence of Lemma~\ref{L3a}.  If $f$ is odd, use the $C^*$-identity to write
\[ \|f(D)\rho(\phi)\|^2 \le \|\rho(\bar{\phi})\| \| |f|^2(D) \rho(\phi) \|. \]
The function $g=|f|^2$ is even, belongs to ${\mathcal S}(R)$ and has Fourier transform supported in $(-2r,2r)$.  Thus, applying Lemma~\ref{L3a} to the function $g$,
\[ \| |f|^2(D) \rho(\phi) \| \le \|\phi\|\sup\{ |f(\lambda)|^2 : |\lambda|\ge a \} \]
and so we obtain (on taking the square root)
\[ \|f(D) \rho(\phi)\| \le \|\phi\|\sup\{ |f(\lambda)| : |\lambda|\ge a \}, \]
which gives the desired result for odd $f$.  The general result is obtained by writing $f$ as a sum of even and odd components (this decomposition accounts for the extra factor of 2 in the statement of Lemma~\ref{L3}). \end{proof}

Using this, let us complete the proof of Lemma~\ref{L2}.   Let $f$ be as in that lemma, and let $\epsilon>0$ be given.  There exists a smooth function $g$ with compactly supported Fourier transform such that $\sup \{|g(\lambda)-f(\lambda)|: \lambda\in\R\} < \epsilon$.  In particular, $|g(\lambda)|<\epsilon$ for $|\lambda|>a$.  Let $r$ be such that $\Supp(\hat{g})\subseteq (-r,r)$ and let $\psi\colon X\to [0,1]$ be a continuous function equal to 1 on a $4r$-neighborhood of $Z$ and vanishing off a $5r$-neighborhood of $Z$.  Write
 \begin{multline*}
f(D) =  \rho(\psi)g(D)\rho(\psi) +\\
 + \rho(1-\psi)g(D)\rho(\psi) + g(D)\rho(1-\psi) + (f(D)-g(D)).
\end{multline*}
The first term is a locally compact operator supported near $Z$, the second and third terms have norm bounded by $2\epsilon$ by lemma~\ref{L3}, and the fourth term has norm bounded by $\epsilon$ by the spectral theorem.  Thus, $f(D)$ lies within $5\epsilon$ of a locally compact operator supported near $Z$.  Since $\epsilon$ is arbitrary, $f(D)\in C^*(Z\subseteq X)$, as was to be shown.

\section{Vanishing results}
As a consequence of the discussion above, if the curvature operator $R$ is uniformly positive outside $Z$, and if the $K$-theory map $K_*(C^*(Z))\to K_*(C^*(X))$ is zero, then the index $\Index(D)\in K_*(C^*(X))$ must vanish.  The usual vanishing theorem establishes this result when $Z=\emptyset$ (i.e., when we have uniformly positive curvature on the whole of $M$), so we can regard these sort of results as a generalization where one allows a ``small amount'' of non-positive curvature.

For example, we have

\begin{proposition} Let $M$ be a complete connected noncompact Riemannian manifold, and let $D$ be a Dirac-type operator whose associated curvature $R$ is uniformly positive outside a compact set.  Then $\Index(D)=0$. \end{proposition}

\begin{proof} Let $K$ be a compact set outside which the curvature is uniformly positive, and let $Z$ be the union of $K$ and a geodesic ray from one of its points to infinity.  The index of $D$ then lies in the image of $K_*(C^*(Z))\to K_*(C^*(X))$ by the discussion above.  But $Z$ is coarsely equivalent to $\R^+$, so $K_*(C^*(Z))=0$. \end{proof}

For another example, imagine that we are in the situation of the ``partitioned manifold index theorem'' of ~\cite{JR6}.  So, let $M$ be a non-compact manifold that is partitioned by a compact hypersurface $N$, which (say) is spin and of non zero $\Ahat$-genus, into two pieces $M^+$ and $M^-$.

\begin{proposition} A partitioned manifold as described above admits no complete metric that has uniformly positive scalar curvature on just one of the partition components ($M^+$ or $M^-$). \end{proposition}

\begin{proof} Suppose $M$ has such a metric.  Using the distance from $N$, construct a proper, coarse map $g\colon M\to\R$ that induces the given partition.  By definition, the partitioned manifold index is
\[ g_*(\Index D) \in K_1(C^*(|\R|))=\Z, \]
and the index theorem of~\cite{JR6} equates this to the $\Ahat$-genus of $N$. Now suppose that $M$ has positive scalar curvature over $M^+$.   Then by our main result, the coarse index factors through $K_1(C^*(M^-\subseteq M))$.  But considering the commutative diagram
\[\xymatrix{ K_1(C^*(M^-\subseteq M))\ar[r]\ar[d]^{g_*} & K_1(C^*(M))\ar[d]^{g_*}\\
K_1(C^*(\R^-\subseteq \R))\ar[r]  & K_1(C^*(\R)) }\]
and noting that the bottom left-hand group is zero, we see that the coarse index vanishes. \end{proof}

\section{The relative index theorem}

The key technical result of~\cite[Chapter 4]{GrLa} is a relative index theorem which may be expressed as follows.

Suppose that $M_1$ and $M_2$ are complete Riemannian manifolds equipped with generalized Dirac operators $D_1$ and $D_2$ respectively, acting on (graded) Clifford bundles $S_1$ and $S_2$.  Suppose further that these items \emph{agree near infinity}: in other words, that there exist compact sets $Z_i\subseteq M_i$ an isometry $h\colon M_1\setminus Z_1\to M_2\setminus Z_2$ that is covered by a bundle isomorphism from $S_1$ to $S_2$, and that this isomorphism conjugates $D_1$ to $D_2$.

In these circumstances once can define a \emph{relative topological index} $\Index_r(D_1,D_2)\in\Z$.  There are several ways to define this quantity. For instance, one can compactify each of the $M_i$ identically outside $Z_i$ (thus obtaining \emph{compact} manifolds $\widetilde{M}_i$ with elliptic operators $\widetilde{D}_i$) and then take the difference of the ordinary Fredholm indices, $\Index(\widetilde{D}_1) - \Index(\widetilde{D}_2)$, to define the relative index.    Alternatively, one can take the Chern-Weil forms ${\mathfrak a}_i$ that are the representatives of the indices of $D_i$ according to the local index theorem, and ``integrate their difference'' over $M_1\cup M_2$: specifically, note that $h^*$ takes ${\mathfrak a}_2$ to ${\mathfrak a}_1$, so that if we let $\mathfrak{a}$ be any smooth form on $M_2$, supported outside $Z_2$ and agreeing with ${\mathfrak a}_2$ near infinity, then the difference
\[ \int_{M_1} ({\mathfrak a}_1 - h^*{\mathfrak a}) - \int_{M_2} ({\mathfrak a}_2-{\mathfrak a}) \]
is well-defined (the integrands are compactly supported) and independent of the choice of $\mathfrak a$, and may be taken as the definition of the ``integral of the difference of Chern-Weil forms''.  The equality of these two definitions of relative index is essentially Proposition 4.6 of~\cite{GrLa}: it shows both that the first definition is independent of the choice of compactification, and that the second definition yields an integer.

\begin{remark}\label{itslocal} Either definition implies that the relative index $\Index_r(D_1,D_2)$ depends only on the geometry of $M_1$ and $M_2$ (and the associated operators) in a neighborhood of the ``regions of disagreement'' $Z_1$ and $Z_2$.  This stability property of the relative index is the basis for several calculations in~\cite{GrLa}. \end{remark}

Now suppose further that $D_1$ and $D_2$ have uniformly positive Weitzenbock curvature operators at infinity.  Then $D_1$ and $D_2$, individually, are Fredholm operators, by Theorem~3.2 of~\cite{GrLa} (a special case of our Theorem~\ref{LCA}). The relative index theorem then states

\begin{proposition} \cite[Theorem 4.18]{GrLa} In the circumstances described above one has
\[ \Index(D_1) - \Index(D_2) = \Index_r(D_1,D_2). \]
\end{proposition}

We are going to generalize this result by allowing the ``regions of disagreement'' $Z_i$ to be non-compact.  The first thing that we need to do is to \emph{define} the relative index in this case.  The following discussion, which is based on the ideas of~\cite{JR12}, leads up to the generalized definition of the relative index, Definition~\ref{relinddef}.

    Let $M_1$ and $M_2$ be complete Riemannian manifolds (as above) and let $D_1$ and $D_2$ be generalized Dirac operators.  Suppose that  $M_1$ and $M_2$ are equipped with coarse maps $q_1$ and $q_2$ to a \emph{control space}
 $X$ (a proper metric space), and that  $Z$ is a subset of $X$.  Put $Z_i = \overline{q_i^{-1}(Z)}\subseteq M_i$ for $i=1,2$.  Suppose that there is a diffeomorphism $h\colon M_1\setminus Z_1 \to M_2\setminus Z_2$ which is covered by an isomorphism of Clifford bundles and Dirac operators and which is \emph{compatible with the control maps} in the sense that $q_1 =q_2\circ f$.

From these data one can  define a relative index in $K_j(C^*(Z))$.  Let $H_i$ be the Hilbert space $L^2(M_i; S_i)$ and regard each $H_i$ as an $X$-module via the control map $q_i$.  In this way we obtain translation algebras $C^*(X;H_i)$, $i=1,2$, each of which contains an ideal $C^*(Z\subseteq X;H_i)$ corresponding to $Z$.  The isometry $h$ between the $M_i$ outside $Z_i$ passes to an unitary isomorphism $V$ between the $L^2(M_i\setminus Z_i; S_i)$, and it is easy to see that conjugation by this unitary induces an isomorphism of quotient $C^*$-algebras
\[ \Phi \colon C^*(M_1;H_1)/C^*(Z_1\subseteq M_1;H_1) \to C^*(M_2;H_2)/C^*(Z_2\subseteq M_2; H_2). \]

\begin{lemma}\label{rel1} Let notation be as above and let $f\in C_0(\R)$.   Then
\[ \Phi[ f(D_1) ] = [f(D_2)] , \]
in the quotient algebra $C^*(M_2;H_2)/C^*(Z_2\subseteq M_2; H_2)$. \end{lemma}

There is also a ``$D^*$-version'' of this discussion.  Namely, following the forthcoming PhD thesis of Paul Siegel~\cite{siegel}we can define ideals $D^*(Z_i\subseteq M_i;H_i)$ as the closure of the finite propagation, pseudolocal\footnote{``Finite propagation'' is defined with respect to the control space $X$ via the control maps $q_i$; ``pseudolocal'' is defined with respect to the ambient manifold $M_i$.} operators that are supported near $Z_i$ and are locally compact on $M_i\setminus Z_i$.  Once again, conjugation by $U$ induces
an isomorphism of quotient $C^*$-algebras
\[ \Psi \colon D^*(M_1;H_1)/D^*(Z_1\subseteq M_1;H_1) \to D^*(M_2;H_2)/D^*(Z_2\subseteq M_2; H_2). \]

\begin{lemma}\label{rel2} Let notation be as above and let $\chi$ be a normalizing function.   Then
\[ \Psi[ \chi(D_1) ] = [\chi(D_2)] , \]
in the quotient algebra $D^*(M_2;H_2)/D^*(Z_2\subseteq M_2; H_2)$. \end{lemma}

\begin{proof}  The proofs of both Lemmas~\ref{rel1} and~\ref{rel2} rely on the finite propagation speed method.  First we give the proof for~\ref{rel1}.  Suppose that $f\in{\mathcal S}(\R)$ and has Fourier transform $\hat{f}$ supported in $(-r,r)$.   As usual, we write
 \[ f(D) = \frac{1}{2\pi} \int_{-\infty}^\infty \hat{f}(t) e^{itD}\,dt \]
and use the fact that, for a Dirac-type operator $D$, $e^{itD}$ has propagation $|t|$.
 Let $\psi_i\colon M_i\to [0,1]$ be a smooth function equal to 1 on an $r$-neighborhood of $Z_i$ and vanishing off a $2r$-neighborhood of $Z_i$, and such that $\psi_1=\psi_2\circ  h$ on $M_1\setminus Z_1$.  Write
\[ f(D_i) = f(D_i)\rho(\psi_i) + f(D_i)(1-\rho(\psi_i)). \]
Since $f(D_i)$ has propagation $r$, the first term belongs to $C^*(Z_i\subseteq H_i)$.  By finite propagation speed we have
\[ V^*e^{itD_1}\rho(1-\psi_1)V = e^{itD_2}\rho(1-\psi_2)\quad\mbox{for $|t|<r$.}\]
Consequently,
\[ V^*f(D_1)(1-\rho(\psi_1))V = f(D_2)(1-\rho(\psi_2))\]
and the proof is complete for $f$ having compactly supported Fourier transform.  The general result follows, since such $f$ are norm-dense in $C_0(\R)$.

The proof of Lemma~\ref{rel2} follows a similar pattern, where the Fourier transform $\hat{\chi}$ must now be understood as a distribution with a mild singularity at 0.  The only additional argument that is needed is to show that \begin{equation}\label{eq3}(V\chi(D_1)V^*-\chi(D_2))\rho(\phi)\end{equation} is   compact for $\phi\in C_0(M_2\setminus Z_2)$.  Suppose in fact that $\phi$ is compactly supported.  Then there is a constant $r>0$ such that $d(Z_2,\Supp(\phi))>r$ and, if we should choose the normalizing function $\chi$ to have Fourier transform supported in $(-r,r)$, then finite propagation speed shows that the displayed quantity in~\ref{eq3} is not just compact --- it is actually \emph{zero}!  The general case follows from this particular one, since any two normalizing functions differ by some $g\in C_0(\R)$, and we already know that for such $g$, the \emph{individual} terms $g(D_1)$ and $g(D_2)$ are locally compact.
   \end{proof}

Now let $\pi_i$ denote the quotient map $C^*(M_i)\to C^*(M_i)/C^*(Z_i\subseteq M_i)$ or $D^*(M_i)\to D^*(M_i)/D^*(Z_i\subseteq M_i)$ as appropriate.  Let us define $A$ to be the pull-back $C^*$-algebra
\[ A = \{(T_1,T_2)\in C^*(M_1;H_1)\oplus C^*(M_2;H_2): \Phi(\pi_1(T_1))=\pi_2(T_2) \}. \]
Similarly define $B$ to be the pull-back $C^*$-algebra
\[ B = \{(T_1,T_2)\in D^*(M_1;H_1)\oplus D^*(M_2;H_2): \Phi(\pi_1(T_1))=\pi_2(T_2) \}. \]
Then $A$ is an ideal in $B$.  Let $D$ denote the Dirac operator on the disjoint union $M_1\sqcup M_2$. Lemmas~\ref{rel1} and~\ref{rel2} above show that for a normalizing function $\chi$, the operator $\chi(D)$ is an element of $B$, and that for a function $f\in C_0(\R)$, the operator $f(D)$ is an element of the ideal $A$.  Consequently there is defined an index of $D$
\begin{equation}\label{eq1} \Index_Z(D) \in K_j(A). \end{equation}
The group $K_j(A)$ can be decomposed as a direct sum.  In fact, let $U\colon H_1\to H_2$ be a \emph{covering isometry} for the identity map~\cite[Definition 6.3.9]{HR3} that agrees on $L^2(M_1\setminus Z_1)$ with the isomorphism $ L^2(M_1\setminus Z_1) \to L^2(M_2\setminus Z_2)$ induced by $h$. (The hypothesis that $h$ boundedly commutes with the control maps assures the existence of such an isometry.) Then there is a split short exact sequence
\begin{equation}\label{eq2} 0 \to C^*(Z_1\subseteq M_1) \to A \to C^*(M_2)\to 0, \end{equation}
where the first map is $a \mapsto (a,0)$, the second is $(a_1,a_2)\mapsto a_2$, and the splitting maps $a$ to $(U^*aU,a)$.  From this split short exact sequence we obtain a direct sum decomposition
\[ K_j(A) = K_j(C^*(Z_1\subseteq M_1)) \oplus K_j(C^*(M_2)).\]
\begin{definition}\label{relinddef} The \emph{relative index} of the above data is the component in $K_j(C^*(Z_1\subseteq M_1)) = K_j(C^*(Z))$ of $\Index_Z(D)\in K_j(A)$.  We denote it by $\Index_r(D_1,D_2)$. \end{definition}

The generalization of Gromov-Lawson's relative index theorem is then

\begin{theorem} Let $(M_i,D_i,q_i)$ be a set of relative-index data over $(X,Z)$, with the notation described above.  Suppose that the operators $D_i$ have uniformly positive Weitzenbock curvature operators outside $Z_i$. Then each $D_i$ has a localized coarse index in $K_j(C^*(Z))$, by Theorem~\ref{LCA}, and the identity
\[ \Index_Z(D_1) - \Index_Z(D_2) = \Index_r(D_1,D_2) \]
holds in $K_j(C^*(Z))$.
\end{theorem}

(The case considered by Gromov and Lawson can be recovered by taking $X=\R^+$, $Z=\{0\}$.)

\begin{proof} Let $A$ be the pull-back algebra that we introduced in our definition of the relative index (so that $A$ consists of pairs $(T_1,T_2)$, $T_i\in C^*(M_i)$, that ``agree away from $Z$''.)   Let $J$ be the ideal in $A$ that consists of pairs $(T_1,T_2)$ where each $T_i$ belongs to $C^*(Z_i\subseteq M_i)$; in fact, $J$ is simply the direct sum $C^*(Z_1\subseteq M_1)\oplus C^*(Z_2\subseteq M_2)$.   Let $D$ denote the Dirac operator on $M_1\sqcup M_2$.

Because of the positive curvature away from $Z$ it follows from Lemma~\ref{L2}  that, for $f\in C_0(\R)$, $f(D)$ belongs to the ideal $J$.  Thus, in this case, the index $\Index_Z(D)$ defined in Equation~\ref{eq1} in fact belongs to $ K_j(J) = K_j(C^*(Z))\oplus K_j(C^*(Z))$,]
and it is apparent from the definitions that, in terms of this direct sum decomposition,
\[ \Index_Z(D) =(\Index_Z(D_1),\Index_Z(D_2)). \]
The definition of the relative index tells us to take the component of $\Index_Z(D)$ in $K_j(C^*(Z))$ in the direct sum decomposition coming from the split short exact sequence~\ref{eq2}.
Restricted to $J$, this sequence takes the form
\[ 0\to C^*(Z) \to C^*(Z)\oplus C^*(Z) \to C^*(Z) \to 0, \]
where the first map is inclusion on the first factor, the second is projection on the second factor, and the splitting used is $a\mapsto (a,a)$.  Using this splitting, one finds that the relevant component of $\Index_Z(D) = (\Index_Z(D_1),\Index_Z(D_2))$ is $\Index_Z(D_1)-\Index_D(Z_2)$, as required.
\end{proof}

 As we observed above, it is an important feature of the Gromov-Lawson relative index that it depends only on the geometry of a neighborhood of the ``region of disagreement''.  The corresponding result is also true in our more general context, and is a key to the applications of the relative index concept in~\cite{JR18}.

\begin{proposition}\cite[Theorem 3.12]{JR18} The relative index of Definition~\ref{relinddef} depends only on the geometry of a metric neighborhood of $Z_1$ and $Z_2$ and the operators thereon. \end{proposition}

Notice that this statement is independent of any positive-curvature hypotheses.

\begin{proof} This follows from the results of~\cite{JRPS33}.  In that paper, it is shown that to a set of relative index data (as described in this section), one may associate a \emph{relative homology class} that lies in the $K$-homology group $K_*(Z)$.  Moreover, comparison of the definitions shows that our coarse relative index is simply the image of this relative homology class under the coarse assembly map
\[ A: K_*(Z) \to K_*(C^*(Z)). \]
The result is therefore a consequence of Proposition 4.8 of~\cite{JRPS33}, which states that in fact the relative \emph{homology class} of a set of relative index data depends only on the geometry in a neighborhood of the region of disagreement.  \end{proof}


\begin{thebibliography}{1}

\bibitem{GrLa}
M.~Gromov and H.B. Lawson.
\newblock Positive scalar curvature and the {Dirac} operator.
\newblock {\em Publications Math\-\'{e}mat\-iques de l'Institut des Hautes
  \'{E}tudes Scientifiques}, 58:83--196, 1983.

\bibitem{HR3}
N.~Higson and J.~Roe.
\newblock {\em Analytic {$K$-Homology}}.
\newblock Oxford Mathematical Monographs. Oxford University Press, Oxford,
  2000.

\bibitem{HRY}
N.~Higson, J.~Roe, and G.~Yu.
\newblock A coarse {M}ayer-{V}ietoris principle.
\newblock {\em Mathematical Proceedings of the Cambridge Philosophical
  Society}, 114:85--97, 1993.

\bibitem{Li}
A.~Lichnerowicz.
\newblock Spineurs harmoniques.
\newblock {\em Comptes Rendus de l'Acad\'{e}\-mie des Sciences de Paris},
  257:7--9, 1963.

\bibitem{JR6}
J.~Roe.
\newblock Partitioning non-compact manifolds and the dual {Toeplitz} problem.
\newblock In D.~Evans and M.~Takesaki, editors, {\em Operator Algebras and
  Applications}, pages 187--228. Cambridge University Press, Cambridge, 1989.

\bibitem{JR12}
J.~Roe.
\newblock A note on the relative index theorem.
\newblock {\em Quarterly Journal of Mathematics}, 42:365--373, 1991.

\bibitem{JR18}
J.~Roe.
\newblock {\em Index Theory, Coarse Geometry, and the Topology of Manifolds},
  volume~90 of {\em CBMS Conference Proceedings}.
\newblock American Mathematical Society, Providence, R.I., 1996.

\bibitem{JRPS33}
John Roe and Paul Siegel.
\newblock Sheaf theory and {Paschke} duality.
\newblock Preprint, 2012. {\tt http://arxiv.org/abs/1210.6420}

\bibitem{siegel}
Paul Siegel.
\newblock {\em Homological calculations with analytic structure groups}.
\newblock PhD thesis, Penn State, 2012.

\end{thebibliography}
\end{document}